\documentclass[12pt]{amsart}
\usepackage[margin=1in]{geometry}
\usepackage{float}
\usepackage{graphicx}
\usepackage{multicol}
\usepackage[table,xcdraw]{xcolor}
\usepackage[colorlinks, citecolor=red, linkcolor=blue]{hyperref}
\usepackage{url}
\usepackage[mathscr]{eucal}
\usepackage{verbatim}
\usepackage{upgreek}
\usepackage{amssymb}
\usepackage{amsmath}
\usepackage{fancyhdr}
\usepackage{url}
\newtheorem{thm}{Theorem}[section]
\newtheorem{prop}[thm]{Proposition}
\newtheorem{lem}[thm]{Lemma}
\newtheorem{cor}[thm]{Corollary}

\theoremstyle{definition}
\newtheorem{question}[thm]{Question}

\newtheorem{definition}[thm]{Definition}
\newtheorem{example}[thm]{Example}

\theoremstyle{remark}

\newtheorem{remark}[thm]{Remark}

\numberwithin{equation}{section}

\begin{document}

\title{On the Uniqueness of Certain Types of Circle Packings on Translation Surfaces}
\author{Nilay Mishra}

\begin{abstract}
Consider a collection of finitely many polygons in $\mathbb C$, such that for each side of each polygon, there exists another side of some polygon in the collection (possibly the same) that is parallel and of equal length. A translation surface is the surface formed by identifying these opposite sides with one another. The $\mathcal{H}(1, 1)$ stratum consists of genus two translation surfaces with two singularities of order one. A circle packing corresponding to a graph $G$ is a configuration of disjoint disks such that each vertex of $G$ corresponds to a circle, two disks are externally tangent if and only if their vertices are connected by an edge in $G$, and $G$ is a triangulation of the surface. It is proven that for certain circle packings on $\mathcal{H}(1, 1)$ translation surfaces, there are only a finite number of ways the packing can vary without changing the contacts graph, if two disks along the slit are fixed in place. These variations can be explicitly characterized using a new concept known as \textit{splitting bigons}. Finally, the uniqueness theorem is generalized to a specific type of translation surfaces with arbitrary genus $g \geq 2$.
\end{abstract}

\maketitle
\markright{ON THE UNIQUENESS OF CERTAIN TYPES OF CIRCLE PACKINGS ON TRANSLATION SURFACES}

\section{Introduction}
Translation surfaces are an interesting facet of mathematics because they interrelate many fields such as topology, differential geometry, complex analysis, and dynamical systems. They are concrete and simple to describe but have many applications in pure math topics such as rational billiards (see \cite{wright2}), geodesic flows (see \cite{massart}), interval exchange transformations (see \cite{yoccoz}), and Teichm\"uller theory (see \cite{fornimatheus}). Important examples of translation surfaces that are an active area of study include square-tiled surfaces (also known as origamis), defined later in \ref{squaredef}, and Veech surfaces. 
\newline
\indent The versatility of translation surfaces comes from the number of different perspectives from which they can be considered. Translation surfaces can be thought of as quotient spaces inheriting some of the topology of $\mathbb C$, as locally Euclidean surfaces with a finite number of singularities, or as Riemann surfaces with an associated differential $1$-form. All of these perspectives will be fruitful in examining them.
\newline
\indent Translation surfaces also have singularities known as cone points, around which the angle is of the form $2\pi \cdot k$, where $k > 1$ is some integer. As such, the surfaces do not have a flat smooth metric, which makes it appealing to investigate them geometrically as well. Indeed, some of the properties of these cone points will make the investigation of circle packings, which are typically considered in the framework of a flat metric, more interesting.
\newline
\indent Circle packings were popularized in the mathematics community through William Thurston's famous \textit{Notes} on the subject (see \cite{thurston}), delivered at a conference in Purdue in 1985. They combine rigidity (tangencies between circles have a particular structure) with flexibility (the radii and locations of circles can vary). They also provide a bridge between the combinatorial nature of triangulations with the geometric nature of surfaces and circles, and have important applications in discrete analytic function theory (see \cite{stephenson}) and conformal uniformization (see \cite{schramm}). Furthermore, according to a conjecture of Thurston that was proven by Rodin and Sullivan in 1987 (see \cite{rodinsullivan}), circle packings can provide a geometric view of the Riemann mapping theorem.
\newline
\indent As per the knowledge of the author, this is the first study of circle packings on translation surfaces with genus $g \geq 2$. Since the metric of such surfaces is not flat and smooth, one needs to generalize the definition of circles, and in the process, circle packings on such surfaces. These new definitions will be presented in \S\ref{sec:back}.
\newline
\indent By providing a framework to unite the two fields of translation surfaces and circle packings, this paper sets up future work at their intersection. Possible approaches for further exploration are given in \S\ref{futdir}. Beyond generalizing the notion of the uniqueness and existence of circle packings to translation surfaces, it will also likely be fruitful to consider corresponding analogs of other circle packing concepts, such as discrete analytic function theory, conformal uniformization, and the Riemann mapping.
\newline
\indent In this paper, we will examine the uniqueness of certain types of circle packings on translation surfaces. In \S\ref{dt}, we will characterize all possible generalized circle configurations on the doubled slit torus and then describe a particular class of circle packings on genus $2$ translation surfaces that will be examined in future sections. In \S\ref{tri}, we will prove several results about a new concept known as \textit{splitting bigons}, which will be used in future sections. We will also generalize the results to surfaces with genus $g \geq 2$. In \S\ref{main}, we will present results about the uniqueness of certain types of circle packings on translation surfaces. In particular, we will show that there are only a finite number of ways that a particular type of circle packings on an $\mathcal{H}(1, 1)$ doubled slit torus can vary if two specific double circles are fixed in place. We will then generalize this uniqueness theorem to genus $g \geq 2$, showing that there are only a finite number of ways for a particular type of circle packings on certain translation surfaces to vary when $2g - 2$ specific double circles are fixed in place.
\section{Background}
\label{sec:back}
Before introducing translation surfaces formally, a few definitions are in order. A \textit{manifold} of real dimension $n$ is a topological space such that every point has an open neighbourhood that maps to an open subset of $\mathbb R^n$ via a homeomorphism known as a \textit{chart}. Given two charts $h_{\alpha}$ and $h_{\beta}$ with overlapping domains $U_{\alpha}$ and $U_{\beta}$, respectively, we define the \textit{transition function} $f_{\alpha, \beta}: h_{\alpha}\left(U_{\alpha} \cap U_{\beta}\right) \rightarrow h_{\beta}\left(U_{\alpha} \cap U_{\beta}\right)$ to be $f_{\alpha, \beta} = h_{\beta} \circ h_{\alpha}$. A \textit{complex manifold} of dimension $n$ has charts $h_{\alpha}: U_{\alpha} \to \mathbb C^n$ such that the transition functions $f_{\alpha, \beta}$ are biholomorphic. A \textit{Riemann surface} is a connected complex manifold of dimension one (or equivalently, real dimension $2$). 
\newline
\indent
Three formal definitions of translation surfaces are now presented from \cite{wright}.
\begin{definition} (First definition of translation surface).
Let $\mathbb C_p$ be the set of all possible finite unions of polygons in $\mathbb C$ such that each side is completely identified with an ``opposite" side that is a translation in the plane of the first. A \textit{translation surface} is an equivalence class of the relation $\sim_t$ on $\mathbb C_p$ such that $\mathfrak{P_1}, \mathfrak{P_2} \in \mathbb C_p$ satisfy $\mathfrak{P_1} \sim_t \mathfrak{P_2}$ if and only if $\mathfrak{P_1}$ can be cut into pieces along straight lines and these pieces can be translated and re-glued to form $\mathfrak{P_2}$. After each cut, the two new boundary segments formed must be identified, and two resulting polygons can be glued together along a pair of edges if and only if these edges are identified.
\end{definition}
\begin{example}
The simplest translation surface is the torus $\mathbb C/\mathbb Z[i]$. 
\end{example}

    \begin{figure}[H]
    \centering {\includegraphics[width=10cm]{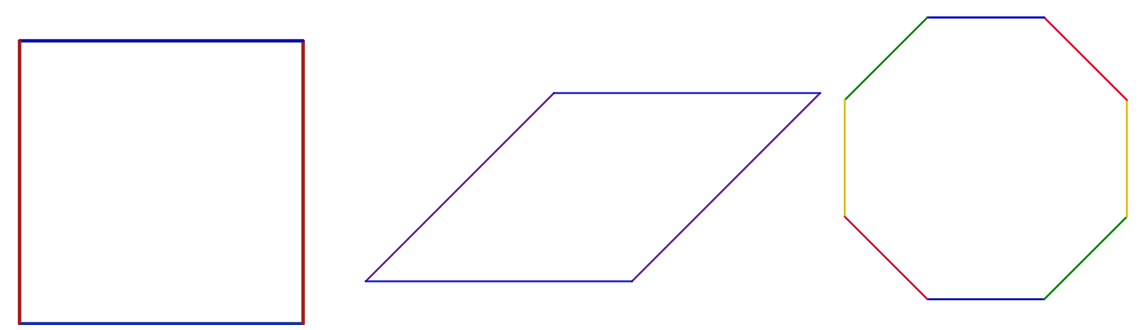}}
    \caption{In each of the three diagrams above, opposite and identified edges have been given the same color. The first two surfaces are tori. All three of the above surfaces are translation surfaces.}
    \label{fig:figure1}
\end{figure}


\begin{definition}
    For a metric space $M$ and an equivalence relation $\sim$ on $M$, a \textit{quotient metric} $d$ on $M/{\sim}$ is defined by:
    \[d([x], [y]) = \inf \left( \sum_{i = 1}^{n} d(p_i, q_i)\right)\]
    where the infimum is taken over sequences $(p_1, p_2, \ldots, p_n)$ and $(q_1, q_2, \ldots, q_n)$ such that $[p_1] = x$, $[q_1] = y$, $[q_i] = [p_{i + 1}]$, and $1 \leq i < n$. 
\end{definition}
\begin{definition}\label{Euclid}
    The \textit{Euclidean} metric of a translation surface in $\mathbb C_p/\sim_t$ is the quotient metric derived from the flat Euclidean metric on $\mathbb C_p \subset \mathbb C$.
\end{definition}
\begin{definition}
    A \textit{saddle connection} on a translation surface is a geodesic for the associated Euclidean metric going from one singularity to another, such that no other singularities lie on the segment.
\end{definition}
There is also a second way of looking at such a surface.
\begin{definition} (Second definition of translation surface).
A \textit{translation surface} is a closed topological surface $X$, together with a finite set of points $\Sigma$ (called \textit{singularities} or \textit{cone points}) and an atlas of charts to $\mathbb C$ on $X \setminus \Sigma$ whose transition maps are translations, such that at each point $p \in \Sigma$ there is some $k > 0$ (called the \textit{degree} or \textit{order} of $p$) and a homeomorphism of a neighborhood of $p$ to a neighborhood of the origin in the $2k+2$ half plane construction that is an isometry away from $p$, depicted in Figure \ref{fig:figure2}.
\end{definition}
\begin{figure}[H]
    \centering
    \includegraphics[width=10cm]{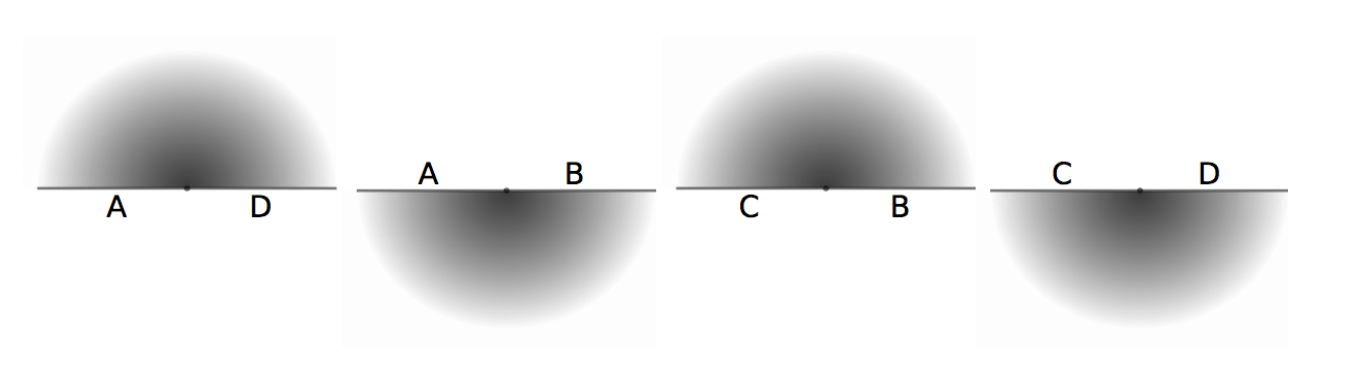}
    \caption{The construction involves taking $k + 1$ copies of the upper half plane and $k + 1$ copies of the lower half plane with the usual flat metric, and gluing them along the infinite rays $[0, \infty)$ and $(-\infty, 0]$ in alternating order. The image above is taken from \cite{wright} and depicts the case for $k = 1$.}
    \label{fig:figure2}
\end{figure}
\begin{remark}
For any translation surface with genus $g > 1$, cone points arise from the fact that no such surface can admit a flat smooth metric (by the Gauss-Bonnet theorem).
\end{remark}
\begin{definition}
The \textit{angle} around a point $p$ in a translation surface is $2\pi$ if $p \in X \setminus \Sigma$ and $2\pi \cdot (k + 1)$ if $p \in \Sigma$ has degree $k$.
\end{definition}
    \begin{example}
In the octagon in Figure \ref{fig:figure1}, there is one singularity of order $2$ and angle $6\pi$.
\end{example}
\begin{prop}
Let the $n$ cone points have degree $d_1, d_2, \ldots, d_n$. Then, we have:
\[\sum_{i = 1}^{n} d_i = 2g - 2\]
where $g$ is the genus of the translation surface.
\end{prop}
\begin{definition}
    Let $g > 1$ and consider a partition $\kappa$ of $2g - 2$. A \textit{stratum} $\mathcal{H}(\kappa)$ is defined to be a collection of translation surfaces such that the order of each cone point is given by $\kappa$.
\end{definition}
\begin{example}
Suppose $g = 2$. Then the two partitions $\kappa$ are given by $(2)$ and $(1, 1)$. Elements of the stratum $\mathcal{H}(2)$ have one cone point of degree $2$ and elements of the stratum $\mathcal{H}(1, 1)$ have two cone points each of degree $1$. 
The regular octagon is a member of $\mathcal{H}(2)$ whereas the square-tiled surface in Figure \ref{fig:figure3} is a member of $\mathcal{H}(1, 1)$.
\end{example}
We present a final algebraic definition of translation surfaces as Riemann surfaces. This will be useful in examining the hyperelliptic involution in Proposition \ref{hyper}.
\begin{definition}\label{third} (Third definition of translation surface).
    A \textit{translation surface} is an Abelian differential (global section of cotangent bundle) on a Riemann surface.
\end{definition}
The significance of this definition comes from the following idea from \cite{wright}, which allows one to think about such surfaces algebraically.
\begin{prop}
    If $\omega$ is an Abelian differential on a Riemann surface $X$, and $\Sigma$ is the set of zeros of $\omega$, then $X \setminus \Sigma$ admits an atlas of charts to $\mathbb C$ whose transition maps are translations.
\end{prop}
One important type of translation surface is the square-tiled surface, also called an origami.
\begin{definition}\label{squaredef}
    A \textit{square-tiled} surface is a translation surface obtained from identifying opposite sides of a finite collection of unit squares in $\mathbb R^2$. 
\end{definition}
\begin{figure}[H]
    \centering
    \includegraphics[width=5cm]{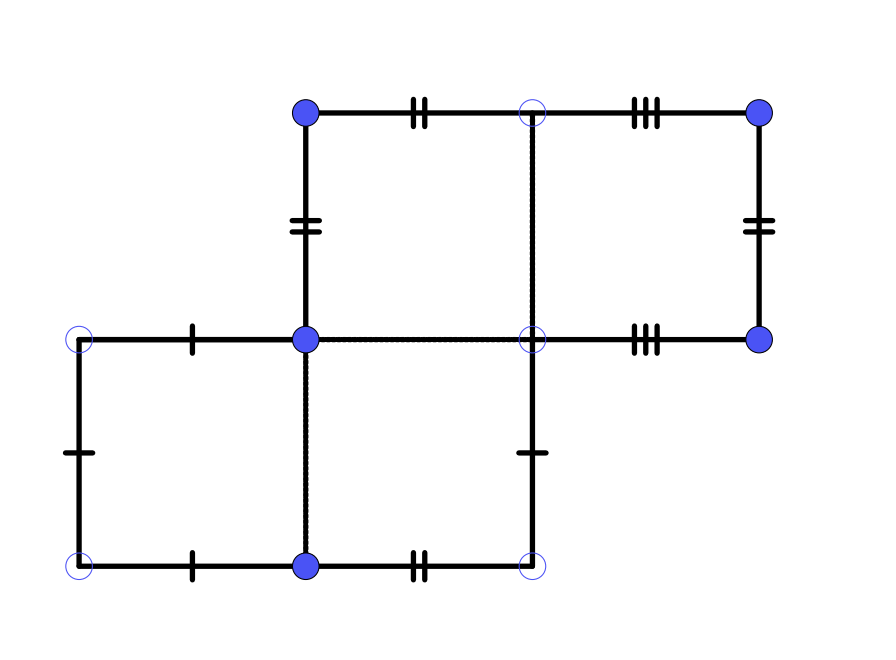}
    \caption{The square-tiled surface shown above has genus $2$ and two singularities, both of order $1$ and cone angle $4\pi$. The image is taken from \cite{massart}.}
    \label{fig:figure3}
\end{figure}
\begin{remark}\label{square}
Many non-square-tiled translation surfaces, such as a regular $2n$-gon, can be thought of as the ``deformed" versions of square-tiled surfaces in a way that is made precise in \cite{massart}.
\end{remark}
Some important background on the concept of circle packings will now be introduced.
\begin{definition}
    A \textit{triangulation} (simplicial $2$-complex) of a surface $S$ is a locally finite decomposition of $S$ into a collection of topologically closed triangles such that any two are either disjoint, intersect at either one or two vertices, or intersect at a single edge.
\end{definition}
Triangulations are permitted to be degenerate in nature, including bigons and loops.
\begin{definition}\label{bigondef}
    A \textit{bigon} in a triangulation consists of two vertices that are connected via two distinct edges.
\end{definition}
\begin{remark} \label{common}
    Two bigons are considered distinct if and only if they share at most one vertex. There will be considerably more to say about bigons in future sections.
\end{remark}
For a given finite union of polygons $\mathfrak{P} \in \mathbb C_p$ with boundary $\partial \mathfrak{P}$ and opposite sides identified, denote by $I(x)$ the set of all points in the plane identified with a specific $x \in \mathbb C$.
\begin{definition}
    A \textit{polygonal sector} $p$ is defined to be a closed region in $\mathbb C$ bounded by the sides of $\mathfrak{P}$ and by a single arc with angle $0 < \alpha \leq 2\pi$.
\end{definition}
\begin{definition}
A \textit{collection of polygonal sectors} $P$ is defined to be a finite set of disjoint polygonal sectors in $\mathbb C$.
\end{definition}
\begin{definition}
    A \textit{boundary segment} of a polygonal sector $p \in P$ is a segment in $\mathbb C$ that is one of the finitely many segments that make up the broken line $\partial p \cap \partial \mathfrak{P}$.
\end{definition}
Let $G$ be a graph with $|P|$ vertices where there exists an edge between two polygonal sectors if and only if there exists a boundary segment of one that identifies entirely with a boundary segment of the other. An equivalence relation $\sim$ on $P$ is defined such that $A \sim B$ if and only if $A$ and $B$ are connected in $G$.
\begin{definition}
\label{def:tscp}
    A collection of polygonal sectors is defined to be a \textit{configuration of generalized circles} if each of the following four conditions hold:
    \begin{itemize}
    \item For all $p \in P$, let $x \in \partial p$. Then for all $y \in I(x)$, there exists $q \in P$ with $p \sim q$ and $y \in \partial q$.
    \item For all $p \in P$, let $x \in p - \partial p$. Then for all $y \in I(x)$, there exists $q \in P$ with $p \sim q$ and $y \in q - \partial q$.
    \item For every polygonal sector $A$:
    \[\sum_{A \sim B} \alpha(B) = 2\pi \cdot k\]
    for some positive integer $k$ where $\alpha: P \to [0, 2\pi)$ is a function that takes in a polygonal sector and outputs the angle of its sole arc. The disk $\bigcup \limits_{A \sim B} B$ is called a \textit{$k$-circle}.
    \item Every $k$-circle is also a circle with respect to the quotient metric on the translation surface i.e. there exists a unique point $c$ called the \textit{center} and a nonnegative real $r$ called the \textit{radius} such that for every point $x$ on the circle, $d(x, c) = r$, where $d$ is the Euclidean metric in Definition \ref{Euclid}. 
    \end{itemize}   
\end{definition}
\begin{definition}
    For a circle configuration on a surface, the \textit{contacts graph} is defined to be a graph $G$ with $n$ vertices $v_1, v_2, \ldots, v_n$ corresponding to the generalized circles $c_1, c_2, \ldots, c_n$ such that $v_i$ and $v_j$ are connected if and only if $c_i$ and $c_j$ are externally tangent.
\end{definition}
\begin{definition}
    A \textit{circle packing} on a translation surface is defined to be a configuration of generalized circles on the surface such that the contacts graph is a triangulation of the surface.
\end{definition}
An example of a circle packing on a sphere is illustrated in Figure \ref{fig:figure4}.
\begin{figure}[H]
    \centering
    \includegraphics[width=5cm]{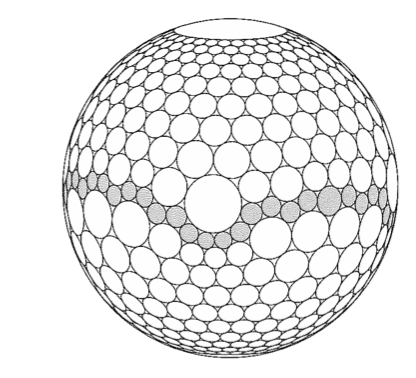}
    \caption{A valid circle packing on a sphere, taken from \cite{stephenson}.}
    \label{fig:figure4}
\end{figure}
The following theorem, from \cite{stephenson}, is quite fundamental to the study of circle packing.
\begin{thm}{(Koebe-Andreev-Thurston, 1936)}.
Let $K$ be a simple planar graph. Then there exists a collection of topological circles $\mathcal{P}_K$ on the Riemann sphere with $K$ as its contacts graph. Furthermore, if $K$ is an triangulation of the Riemann sphere, this circle packing is univalent and unique (up to a M\"obius transformation).
\end{thm}
One extension of the result to a torus, due to \cite{schramm}, is given below.
\begin{thm}(Schramm, 1996). 
Let $T$ be a torus with universal cover $\pi: \mathbb C \to T$. Let $\tilde{G} = (\tilde{V}, \tilde{E})$ be a graph embedded in $\mathbb C$ that is invariant under the deck transformations of $\pi$, and let $G = (V, E) =  \pi(\tilde{G})$. Let $P_v$ be a smooth topological disk corresponding to each $v \in V$. Then there exists a doubly periodic packing $Q$ whose contacts graph is $\tilde{G}$ and has $Q_v$ homothetic to $P_{\pi(v)}$.
\end{thm}
As per \cite{stephenson}, it can also be generalized to any Riemann surface with a flat metric, as follows.
\begin{thm}\label{realdut}
Let $K$ be a triangulation of a compact genus $g$ surface. Then there exists a unique conformal structure on $K$ such that the resulting Riemann surface $S_K$ supports a circle packing on $K$ that fills the surface. The packing is unique up to the conformal automorphisms of $S_K$. 
\end{thm}
The automorphisms of complex tori can be characterized as follows, according to \cite{miranda}.
\begin{prop}\label{deg}
    Suppose that $X = \mathbb C/\Lambda$ is a complex torus and $F: X \to X$ is an automorphism. Then $F$ is induced by a linear map $G(z) = az + b$, where $a, b, \in \mathbb C$ and $a\Lambda = \Lambda$. Additionally, $a$ is a root of unity and has a finite number of possibilities.
\end{prop}
Using this theorem, one can show the following corollary that will be useful for subsequent proofs.
\begin{cor}\label{dut}
Let $P = (P_v, v \in V)$ and $Q = (Q_v, v \in V)$ be two circle packings on the same torus based upon the same triangulation, and suppose that there exist vertices $a, b \in V$ such that $P_a = Q_a$ and $P_b = Q_b$. Then $P$ and $Q$ are identical circle packings (i.e. $P_v = Q_v$ for all vertices $v \in V$).
\end{cor}
\begin{remark}
    The idea behind this proof is that Proposition \ref{deg} implies that the conformal automorphism group for tori has $2$ degrees of freedom (since $b \in \mathbb C$ and $a$ has a finite number of possible values). Thus, fixing two circles on the surface is sufficient to fix all the other circles in the packing in place as well.
\end{remark}
We will now introduce a natural map on genus $2$ Riemann surfaces.
\begin{definition}
    A Riemann surface is called \textit{hyperelliptic} if and only if it is conformally equivalent to the curve formed from the two-valued analytic function:
    \[f(z) = \sqrt{\prod \limits_{i = 1}^{n} (z - a_i)}\]
    for $n > 4$.
\end{definition}
\begin{remark}
    For $n = 1, 2$, the curve formed from $f(z)$ is conformally equivalent to the Riemann sphere. For $n = 3, 4$, the curve formed from $f(z)$ is conformally equivalent to the complex torus.
\end{remark}
Hyperelliptic surfaces are important because of the following result from \cite{bouwman}.
\begin{prop}
    Every hyperelliptic surface $M$ admits a unique conformal involution $\eta$ that has exactly $2g + 2$ fixed points, where $g$ is the genus of $M$. Furthermore, $M/\eta$ is equivalent to the Riemann sphere.
\end{prop}
\begin{definition}
    A \textit{hyperelliptic involution} is the involution associated with a hyperelliptic Riemann surface.
\end{definition}
\begin{prop}\label{hyper} 
Every Riemann surface $M$ of genus $2$ is hyperelliptic (i.e. admits a conformal hyperelliptic involution $\eta: M \to M$ with exactly six fixed points).
\end{prop}
\begin{remark}
    Since every translation surface can be thought of as a Riemann surface with an associated Abelian differential as per Definition \ref{third}, it follows that every genus $2$ translation surface is hyperelliptic. The hyperelliptic involution for a genus $2$ surface provides a natural symmetry of this surface.
\end{remark}
A key result about genus $2$ translation surfaces (from \cite{bouwman}) is that they can be thought of in terms of their saddle connections and the hyperelliptic involution.
\begin{thm}(McMullen, 2007).
Let $M$ be a translation surface of genus $2$ (either in the $\mathcal{H}(1, 1)$ stratum or the $\mathcal{H}(2)$ stratum). Then $M$ contains a saddle connection $J$ such that $J \neq \eta(J)$ and splits along $J \cup \eta(J)$ into the connected sum of two slit tori. Refer to Figure \ref{fig:figure5}.
\end{thm}
\begin{figure}[ht!]
    \centering
    \includegraphics[width=8cm]{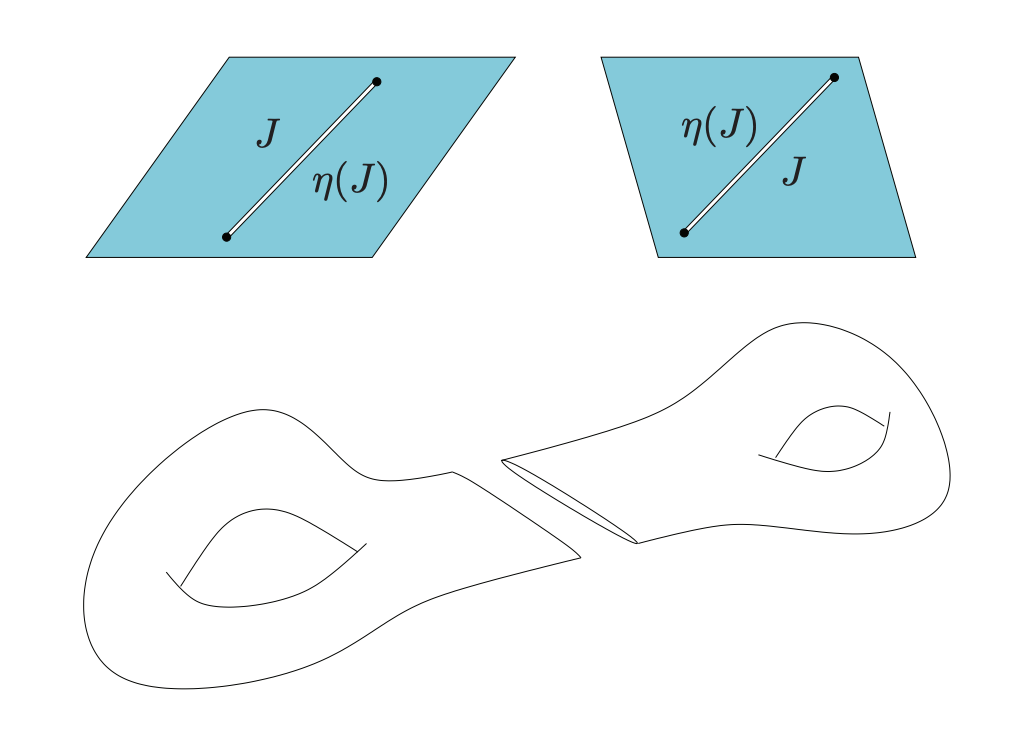}
    \caption{The surface depicted above is an arbitrary genus $2$ surface that has been expressed as the connected sum of two slit tori. It belongs in the $\mathcal{H}(1, 1)$ stratum. In order to express an $\mathcal{H}(2)$ surface as per the theorem, one of the slits needs to start and end at the same point of the torus. This image is taken from \cite{bouwman}.}
    \label{fig:figure5}
\end{figure}
It follows that any surface in $\mathcal{H}(1, 1)$ can be thought of as a doubled slit torus (i.e. the connected sum of two slit tori) in terms of a unique hyperelliptic involution. There is also another way of thinking about $\mathcal{H}(2)$ surfaces (from \cite{wright}).
\begin{prop}
Every genus $2$ translation surface in $\mathcal{H}(2)$ can be decomposed into a cylinder and a single slit torus as per a saddle connection $J \neq \eta(J)$.
\end{prop}

\section{Doubled Slit Torus}\label{dt}
Any translation surface in the $\mathcal{H}(1, 1)$ stratum can be expressed as a doubled slit torus (i.e. the connected sum of two slit tori) in terms of a unique hyperelliptic involution. The following proposition gives a way to characterize all possible circle configurations on a doubled slit torus.
\begin{prop}
Every circle configuration of an $\mathcal{H}(1, 1)$ doubled slit torus can be thought of as a pair $(P, Q)$ of collections of generalized circles on the torus (equivalently, doubly periodic circle configurations on the plane) such that for every generalized circle $C$ in $P$ and $Q$, the slit either intersects $C$ in two distinct points, passes through the center of $C$, or does not intersect $C$ at all.
\end{prop}
\begin{remark}
The center of $C$ is guaranteed to exist as a point by the fourth condition in Definition \ref{def:tscp}. The conditions on the slit also arise from this fourth condition, which requires that the circle respect the metric of the surface.
\end{remark}
\begin{proof}
    Every generalized circle configuration on a doubled slit torus consists of flat circles, whose interior and boundary do not contain any part of the slit, and non-flat circles. The flat circles must be fully contained on one of the slit tori that the doubled slit torus is made out of, and so belongs to either collection $P$ or $Q$ depending on which tori it lies on.
    \newline
    \indent Now, consider an arbitrary non-flat circle $C$. By the fourth condition of Definition \ref{def:tscp}, the center of $C$ must necessarily exist, and the circle must respect the Euclidean metric of the surface, as defined in Definition \ref{Euclid}. Suppose, for the sake of contradiction, that the slit did not pass through the center of $C$ and at least one end of the slit lies in the interior of $C$. Then the center of $C$ must lie in exactly one of the slit tori, and so belongs either in collection $P$ or collection $Q$.
    \newline
    \indent The idea is that $C$ must be a double circle, with one copy on each slit torus. Suppose, for the sake of contradiction that there was only a single copy of $C$ with radius $r$ on one of the slit tori. Let $A$ be the center of $C$. Then there exists a point $B$ on $C$ such that $AB$ does not intersect the slit and $d(A, B) = r$. Now, consider any point $X$ on the portion of the slit contained in the interior of $C$. Connecting the center to $X$ and extending, one can see that it will intersect at a point $Y$ on the boundary of $C$, in the first slit tori but not the second. However, $d(A, Y) > r = d(A, B)$, as one must go around the slit and cannot pass directly through it (since $C$ is a single circle). Therefore, $C$ does not respect the Euclidean metric as desired.
    \newline
    \indent Since the second copy of $C$ on the other slit torus must have a distinct center $A' \neq A$ (since the center $A$ does not lie on the slit), there is a contradiction since Definition \ref{def:tscp} requires there to be exactly one center. Therefore, either the center of $C$ must pass through the slit or the slit must pass entirely through $C$ (see figure \ref{fig:figure7} for an example of the latter case). This completes the proof of the proposition.
\end{proof}
\begin{example}
This gives a way to generate possible circle configurations, and check whether or not the contacts graph forms a triangulation. For example, start with the circle configuration $T$ of the torus, as shown in Figure \ref{fig:figure6}. 
\begin{figure}[H]
    \centering
    \includegraphics[width=4cm]{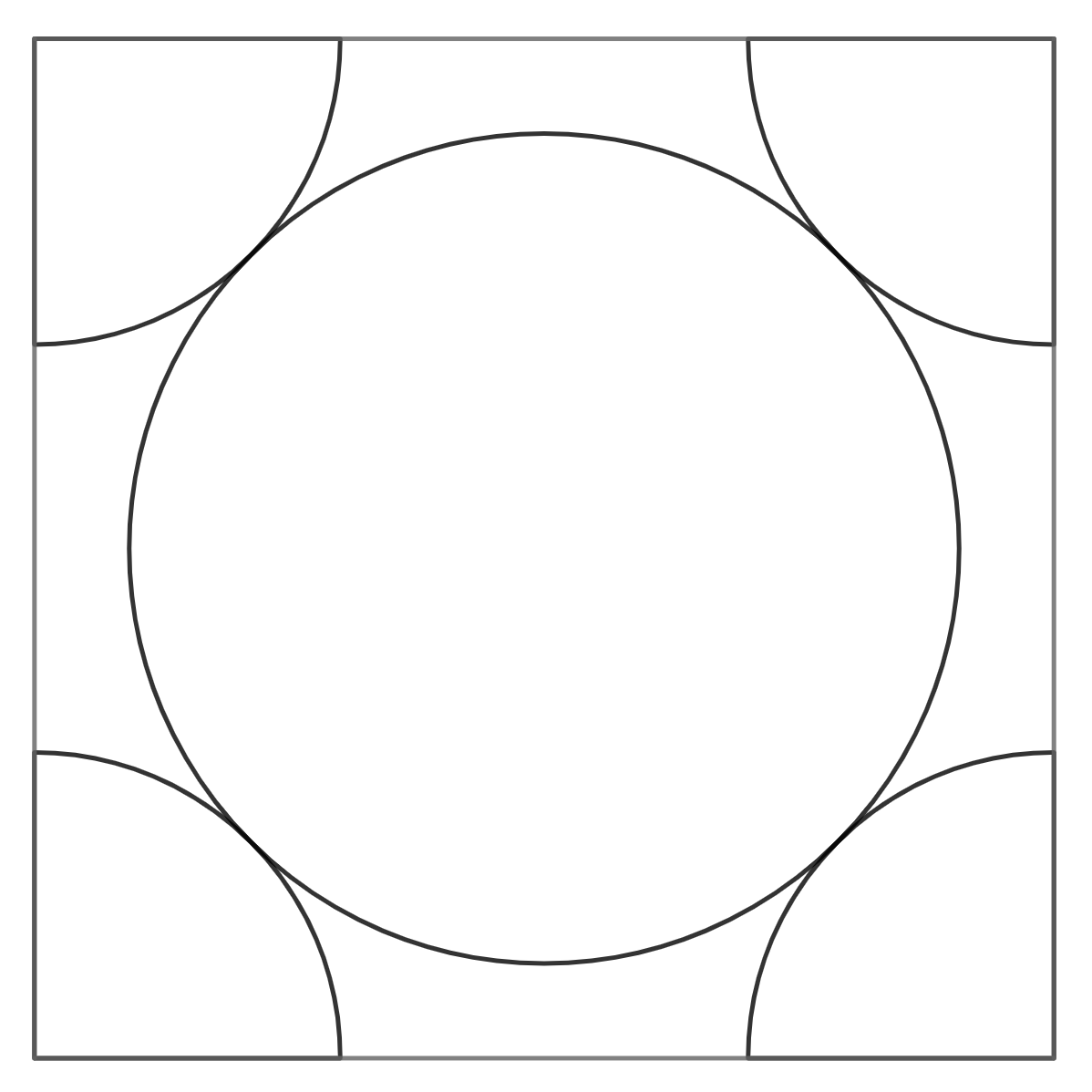}
    \caption{The univalent circle configuration above consists of two circles, and its contacts graph is simply $K_2$.}
    \label{fig:figure6}
\end{figure}
Set $P = Q = T$. The two singularities are constrained by the conditions on the slit. Figures \ref{fig:figure7}-\ref{fig:figure11} illustrate some of the possible locations of the slit, and the resulting circle configuration. Here, regions with the same color comprise a single generalized circle.
\end{example}
\begin{figure}[H]
    \centering
    \includegraphics[width=10cm]{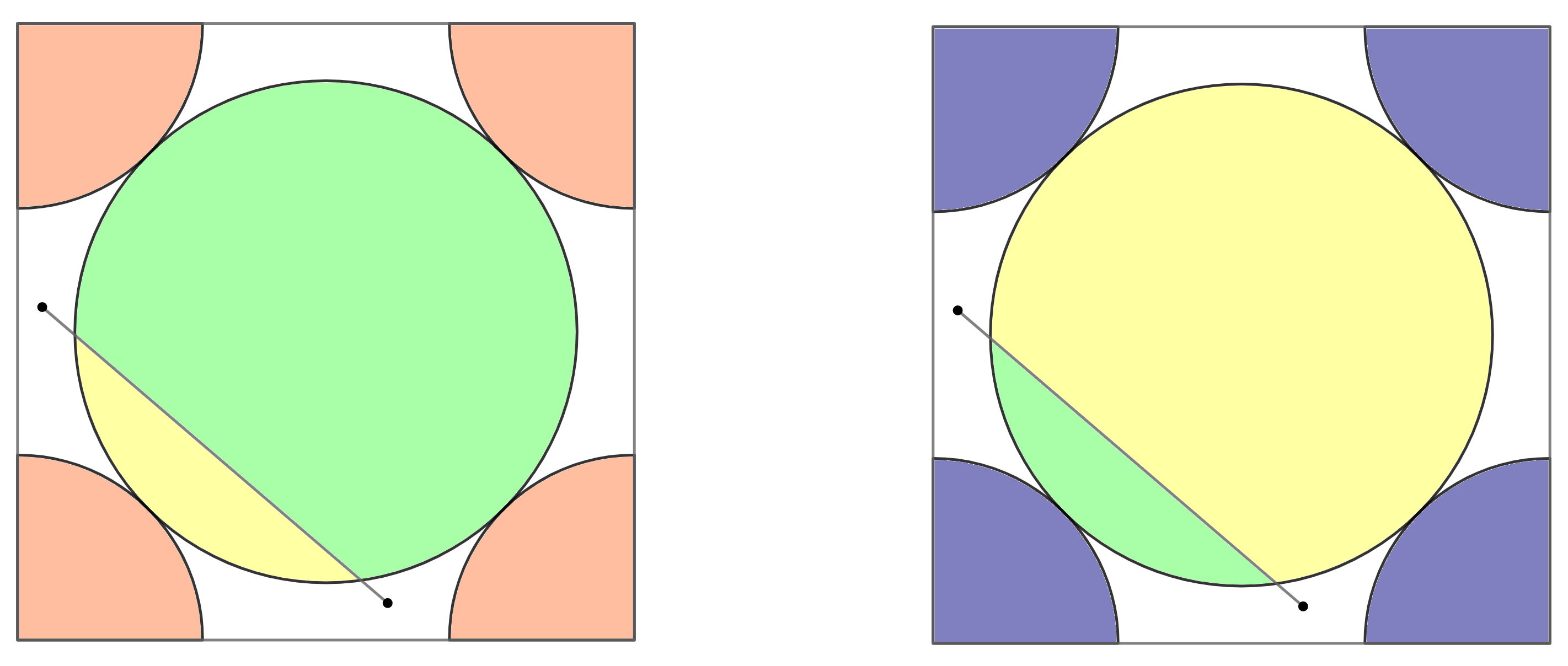}
    \caption{Both the yellow and the green circles are tangent to the orange and purple circles. The slit can be moved without changing the configuration as long as both singularities lie outside the circles and the slit intersects a circle. If the slit does not intersect any circle, the resulting contacts graph will not be connected.}
    \label{fig:figure7}
\end{figure}
\begin{figure}[H]
    \centering
    \includegraphics[width=10cm]{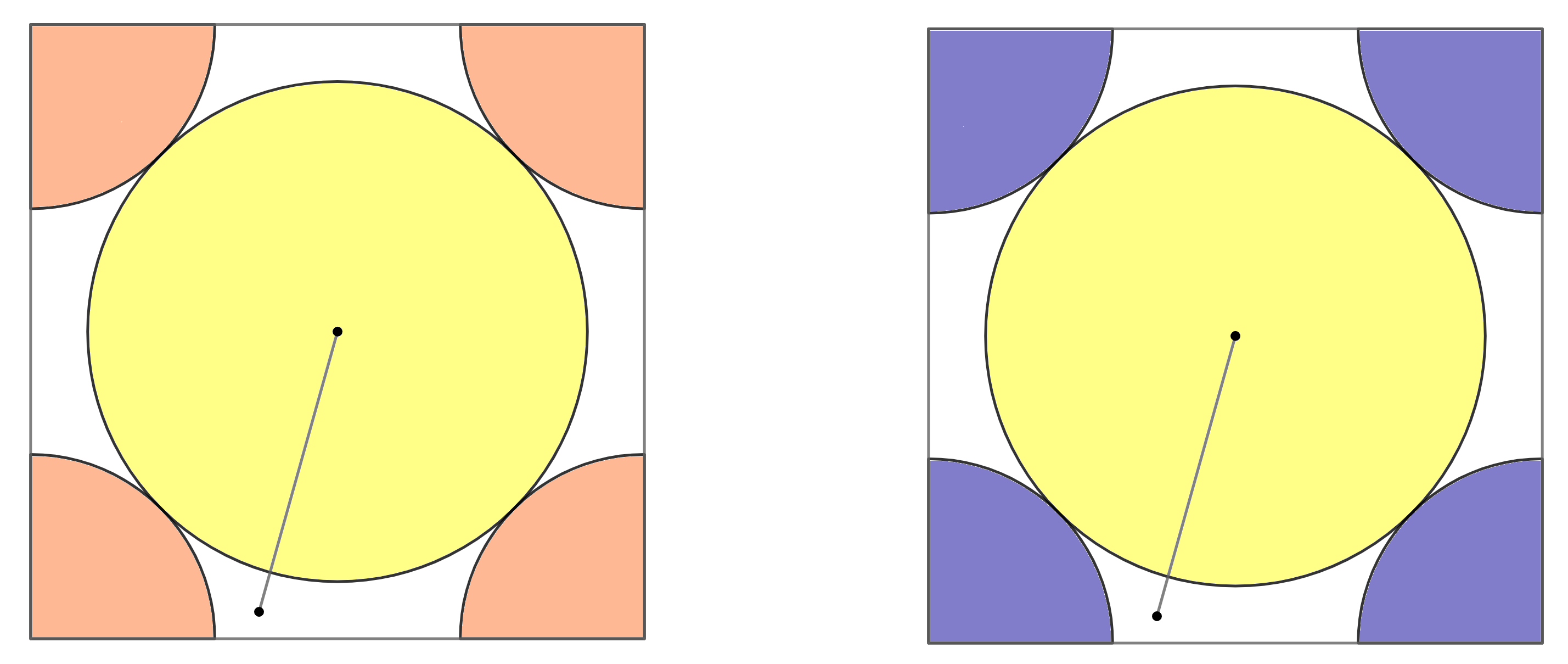}
    \caption{The yellow $2$-circle is tangent to the orange and purple circles. It is not required that the center be a singularity, as illustrated in the diagram. Rather, if a slit does not pass through a circle entirely, it is sufficient for the slit to pass through the center.}
    \label{fig:figure8}
\end{figure}
\begin{figure}[H]
    \centering
    \includegraphics[width=10cm]{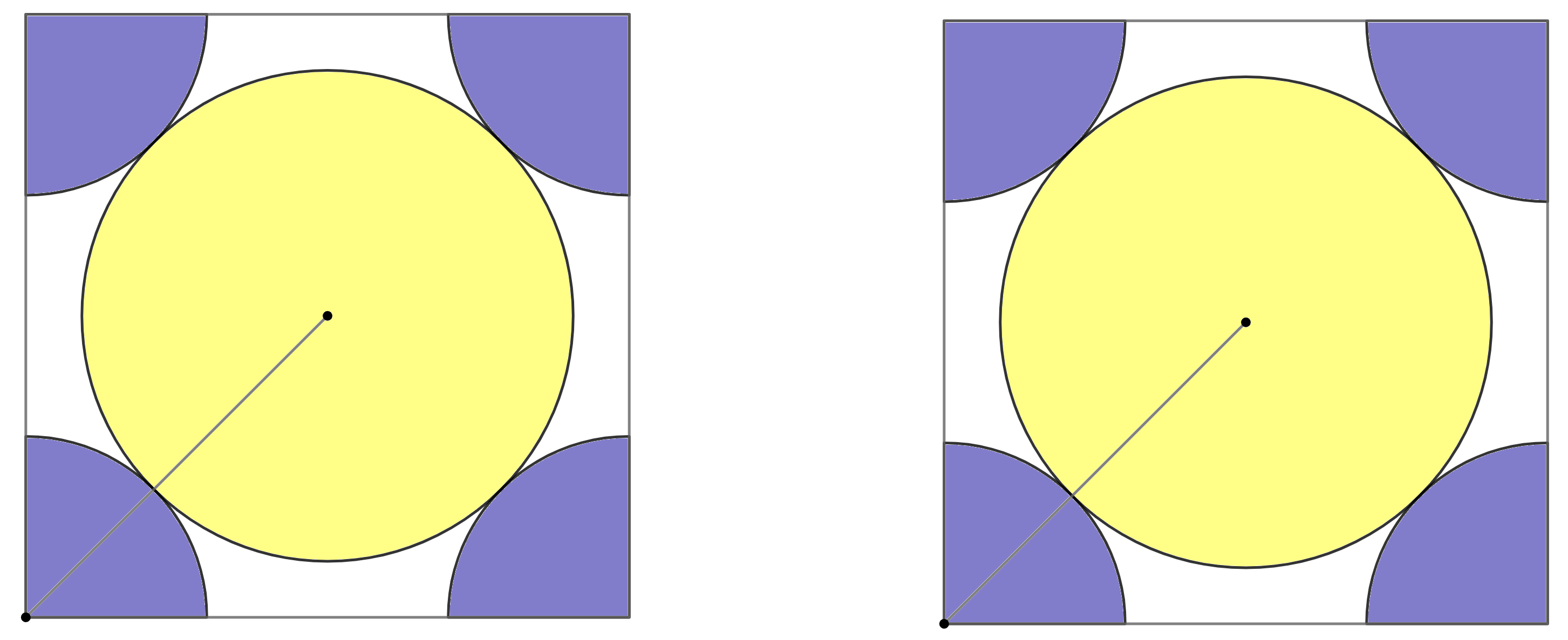}
    \caption{The yellow $2$-circle is tangent to the purple $2$-circle. In this case, one singularity must lie at one of the four corners because the slit must be fully contained in the square and pass through the center of the purple circle. The other singularity can be anywhere along the diagonal inside the yellow circle past the center.}
    \label{fig:figure9}
\end{figure}
\begin{figure}[H]
    \centering
    \includegraphics[width=10cm]{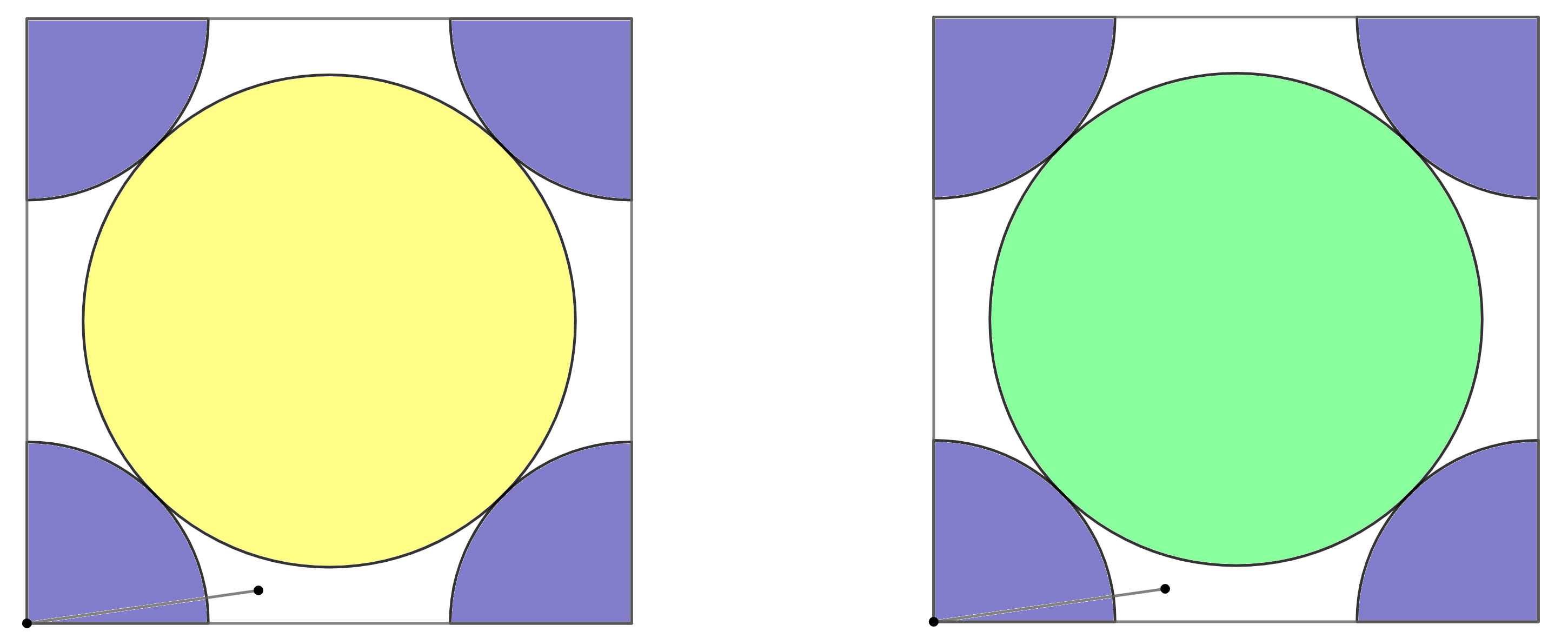}
    \caption{The yellow and green circles are both tangent to the purple $2$-circle. One of the singularity lies at the corner of the square because the slit must pass through the center of the purple circle. In this configuration, the other singularity must lie in one of the two adjacent interstice regions, as shown above.}
    \label{fig:figure10}
\end{figure}
\begin{figure}[H]
    \centering
    \includegraphics[width=10cm]{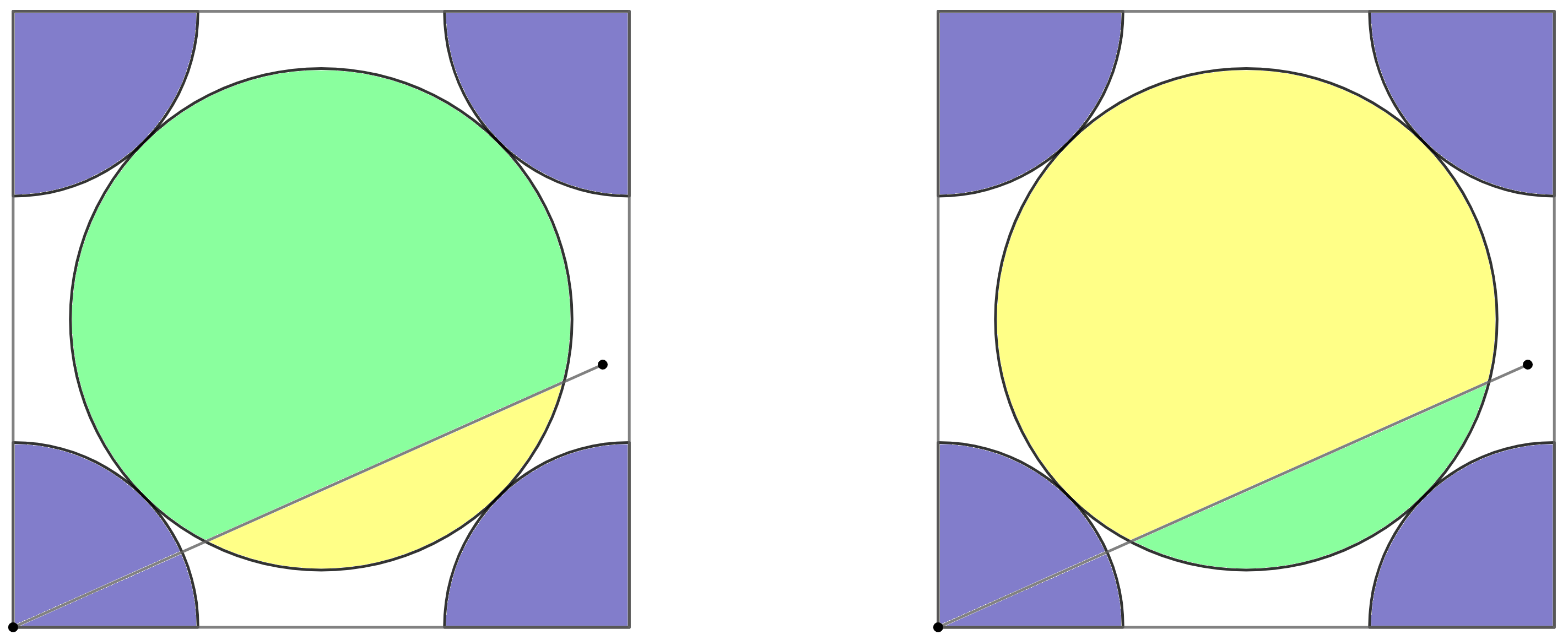}
    \caption{The yellow and green circles are both tangent to the purple $2$-circle. One of the singularity lies at the corner of the square because the slit must pass through the center of the purple circle. In this configuration, the other singularity must lie in one of the two opposite interstice regions, as shown above.}
    \label{fig:figure11}
\end{figure}
The following result also holds.
\begin{prop}\label{triangprop}
Suppose there exists a pair $(P, Q)$ of collections of generalized circles on the torus (equivalently, doubly periodic circle configurations on the plane) such that the contacts graphs of both $P$ and $Q$ form a triangulation on each torus. Furthermore, suppose that both ends of the slit are located at the center of a generalized circle and that every point along the slit either lies on the boundary of or is enclosed inside of a generalized circle. Then the contacts graph of the circle configuration $(P, Q)$ is a triangulation of the $\mathcal{H}(1, 1)$ doubled slit torus.
\end{prop}
\begin{remark}
The key idea behind the proof is that the conditions guarantee that all the interstice regions between the generalized circles are triangular and do not contain singularities.
\end{remark}
\begin{proof}
    It is first checked that all the faces of the contacts graph of $(P, Q)$ are triangles (or degenerate triangles). Clearly all the faces in just $P$ and all the faces in just $Q$ are triangles since $P$ and $Q$ each form a triangulation on their respective torus. Hence, it is necessary to check the faces at least one of whose vertices lies in both $P$ and $Q$ (namely, the slit). 
    \newline
    \indent First, let there be $k$ circles $C_1, C_2, \ldots, C_k$ such that for all $1 \leq i < k$, the double circle $C_i$ is externally tangent to the double circle $C_{i + 1}$. Then the slit goes from the center of $C_1$ to the center of $C_k$ and passes through the centers of each of the circles. This forms $k - 1$ bigons corresponding to the tangencies between the double circles of the from $C_i$ and $C_{i + 1}$. Now, consider the faces formed with at least one vertex on the slit and all the remaining (of which there are at least one) vertices on the same slit torus. Note that these faces are completely contained in a single slit torus, and thus must be contained in either the triangulation formed by $P$ or the triangulation formed by $Q$, as desired. 
    \newline
    \indent We now check the other condition for a triangulation, namely that any two triangles should intersect in an edge, a single vertex, two vertices, or not at all. If both triangles belong to the same slit torus, the result is clear since $P$ and $Q$ each triangulate their respective slit tori. Suppose that one of the triangular faces is one of the $k - 1$ bigons formed by $C_{i}$ and $C_{i + 1}$. Then the other face is either another of those bigons (in which case the intersection is either empty or a single vertex) or the other face is a triangle whose sides are fully contained in one of the slit tori (in which case the intersection is either an edge of the bigon, a single vertex of the bigon, or empty).
    \newline
    \indent Now, suppose that one triangle belongs to one slit tori and the other triangle belongs to the other slit tori, and none of them are one of the $k - 1$ bigons. Then either the triangles do not intersect at all, or they intersect in two points (namely the vertices of one of the $k - 1$ bigons), or they intersect in a single point (the vertex of one of the bigons), or they share an edge (that is a part of one of the bigons). This completes the proof of the proposition. 
\end{proof}
\begin{remark}
    A special case of the above result is when the slit connects the centers of two externally tangent double circles.
\end{remark}
This will be a useful result, as many of the circle packings that will be explored have a triangulation as defined by the previous proposition.
\section{Splitting Bigons}\label{tri}
The topological properties of bigons, as defined in Definition \ref{bigondef}, will be considered as follows.
\begin{definition}
    The \textit{associated loop} defined by a bigon between $V_1$ and $V_2$ can be determined by starting at $V_1$, traversing an edge to $V_2$, and then returning to $V_1$ via the other edge.
\end{definition}
\begin{remark}
The associated loop does not depend on whether one starts traversing at $V_1$ or at $V_2$.
\end{remark}
\begin{definition}
    A \textit{splitting bigon} is a bigon whose associated loop contains exactly two vertices in the triangulation and topologically disconnects the doubled slit torus into two disjoint slit tori.
\end{definition}
Consider the following lemma.
\begin{lem} \label{lemma1}
Let there be an arbitrary splitting bigon between two points $V_1$ and $V_2$ in the triangulation of a doubled slit torus. Then, after removing the vertices $V_1$ and $V_2$, the edges between them, and the loops containing just $V_1$ or just $V_2$, the resulting graph contains exactly two connected components.
\end{lem}
\begin{proof}
Consider two arbitrary vertices $A, B$ of the triangulation such that $A \neq V_1, V_2$ and $B \neq V_1, V_2$. Suppose that the doubled slit torus is divided by the splitting bigon into two different slit tori $\Gamma_1$ and $\Gamma_2$. Then $A, B$ are each part of either $\Gamma_1$ or $\Gamma_2$, since neither lie on the associated loop of the bigon. \newline
\indent
If $A, B \in \Gamma_1$ or $A, B \in \Gamma_2$, then there exists a path along the triangulation from $A$ to $B$ that does not pass through either $V_1$ or $V_2$ (otherwise, the same splitting bigon topologically splits the slit tori $A$ and $B$ are on, which is impossible). If $A$ and $B$ belong to different tori, then any path between them along the edges of the triangulation must contain a point on the bigon (namely, either $V_1$ or $V_2$). Thus, after removing $V_1, V_2$ and all loops and edges containing them, $A$ and $B$ become disconnected. This completes the proof of the lemma.
\end{proof}
Let $B_1$ and $B_2$ be two distinct splitting bigons that are both a part of $T$. Denote the two points in the splitting bigon $B_1$ as $V_1$ and $V_2$. A equivalence relation $\sim$ is set up with exactly two classes among all other vertices in the triangulation such that $X \sim Y$ if and only if there exists a path from $X$ to $Y$ not passing through $V_1$ or $V_2$. The two vertices belonging to the bigon $B_2$ must be part of the same equivalence class under $\sim$ by definition.
\newline
\indent
Consider the following second lemma.
\begin{lem} \label{lemma2}
Suppose that the doubled slit torus is topologically split by the associated loop of the splitting bigon $B_1$ into two disjoint slit tori $\Gamma_1$ and $\Gamma_2$. Every point on the associated loop of the other splitting bigon $B_2$ lies on the same slit torus $\Gamma_i$, where $i \in \{1, 2\}$.
\end{lem}
\begin{proof}
There are two cases, depending on how many vertices the two splitting bigons $B_1$ and $B_2$ have in common. Note that by Remark \ref{common}, they either have one vertex in common or none at all.
\newline
\indent
\textbf{Case 1 (One common vertex).} The two vertices of $B_2$ are either of the form $V_1$ and $V_3$ or of the form $V_2$ and $V_3$, where $V_3 \neq V_1, V_2$.
\newline
\indent
Let $\Gamma_i$ be the slit torus (chosen out of $\Gamma_1, \Gamma_2$) containing $V_3$. Suppose for the sake of contradiction, that there was a point, not necessarily a triangulation vertex, $X \in B_2$ on $\Gamma_{3 - i}$ (the other slit torus). Then, any path along the associated loop of $B_2$ from $X$ to $V_3$ must have passed through a point on the associated loop of $B_1$. If it passed through $V_1$, there is a completed loop $V_1$ to $X$ to $V_1$ before $V_3$ is reached, and so it is considered a separate bigon between $V_1$ and $X$. If it passed through $V_2$, the bigon would contain three distinct points on the triangulation $(V_1, V_2, V_3)$, which contradicts the given. There can not be any other point of intersection $Y \neq V_1, V_2$ between $B_1$ and $B_2$ or else $Y$ would be a vertex of the triangulation $T$, which is a contradiction since both $B_1$ and $B_2$ contain exactly two vertices in $T$.
\newline
\indent
\textbf{Case 2 (No common vertex).} The two vertices defining $B_2$ are of the form $V_3$ and $V_4$ such that $\{V_1, V_2\}$ and $\{V_3, V_4\}$ are disjoint.
\newline
\indent
If $V_3$ and $V_4$ belonged to distinct tori, then $B_2$ and $B_1$ would necessarily intersect at some point $X$, which would belong to $T$. This is a contradiction since both $B_1$ and $B_2$ contain exactly two vertices in $T$. Thus, $V_3$ and $V_4$ must belong to the same slit torus. Call this slit torus $\Gamma_i$ where $i \in \{1, 2\}$. Suppose for the sake of contradiction, that there was a point $X \in B_2$ on $\Gamma_{3 - i}$ (the other torus). Then, any path from $X$ to $V_4$ would intersect $B_1$ at either $V_1$ or $V_2$. Without loss of generality, suppose it intersects at $V_1$. This is a contradiction since $B_2$ cannot contain three distinct vertices of $T$ (namely $V_3, V_4,$ and $V_1$). This completes the proof of the lemma.
\end{proof}
Consider the following result.
\begin{prop} \label{bigon}
Let $T$ be a triangulation of an $\mathcal{H}(1, 1)$ doubled slit torus. If $B_1$ and $B_2$ are two distinct splitting bigons that are both a part of $T$, then the associated loops of $B_1$ and $B_2$ are topologically equivalent.
\end{prop}
\begin{proof}
Divide the doubled slit torus along $B_1$ to form two slit tori $\Gamma_1$ and $\Gamma_2$, as before. By Lemma \ref{lemma2}, there exists a slit torus $\Gamma_i$ for $1 \leq i \leq 2$ such that every point on the associated loop of $B_2$ lies on $\Gamma_i$. This allows one to cut the surface $\Gamma_i$ along $B_2$. If $\Gamma_i$ did not disconnect, then $B_2$ is unable to disconnect the entire doubled slit torus as well, which contradicts the fact that $B_2$ is a splitting bigon. Therefore, cutting along $B_2$ disconnects $\Gamma_i$ into two surfaces $S_1$ and $S_2$ with disjoint interiors such that $S_1 \# S_2 = \Gamma_i$. Without loss of generality, denote the surface containing the points on the associated loop of $B_1$ to be $S_1$ and the other one to be $S_2$.
\newline
\indent
Then the genus of $S_2$ is necessarily $1$ as $B_2$ divides the double torus into two slit tori since it is a splitting bigon. Since the genus of $\Gamma_i$ is also $1$ and $S_1 \# S_2 = \Gamma_i$, it follows that the genus of $S_1$ is $0$. Thus, $S_1$ is homeomorphic to a sphere and has a trivial fundamental group. Since the associated loops of both $B_1$ and $B_2$ fully lie on $S_1$, the loops must be topologically equivalent. This completes the proof of the proposition.
\end{proof}
Color the two holes of a doubled slit torus red and blue in some arbitrary order. Consider the following lemma.
\begin{lem} \label{order}
    Suppose that there are $k$ distinct splitting bigons. Then, there exists a way to order them as $B_1, B_2, \ldots, B_k$ such that for all $1 \leq i \leq k$, $B_i$ divides the doubled slit torus into two slit tori $T_{i, 1}$ and $T_{i, 2}$ with disjoint interiors such that $T_{i, 1}$ contains the red hole, $T_{i, 2}$ contains the blue hole, and there exist $i - 1$ other splitting bigons whose associated loop is completely contained in $T_{i, 1}$.   
\end{lem}
\begin{proof}
    This result will be proven using induction on $i$. 
    \newline
    \indent
    \textbf{Base Case.} Suppose that $i = 1$. It will be shown that there exists a splitting bigon $B_1$ such that $T_{1, 1}$ does not contain the associated loop of any of the other splitting bigons. Suppose, for the sake of contradiction, that no such $B_1$ exists. Then, for all splitting bigons $X$ that split the doubled slit torus into $T_1$ and $T_2$ with $T_1$ containing the red hole, there exists some other splitting bigon $X_1$ whose associated loop lies completely in $T_1$. Similarly, applying this idea to $X_1$, it can be deduced that there must exist an $X_2$ whose associated loop is even closer to the red hole. Continuing this ad infinitum, there must exist an infinite sequence $X_1, X_2, \ldots$ of distinct splitting bigons, which contradicts the finiteness of the total number of splitting bigons.
    \newline
    \indent
    \textbf{Inductive Hypothesis.} Suppose that for some $1 \leq i \leq k - 1$, there exists a splitting bigon $B_i$ that divides the the doubled slit torus into two slit tori $T_{i, 1}$ and $T_{i, 2}$ such that $T_{i, 1}$ contains the red hole and there exist $i - 1$ other splitting bigons whose associated loop is completely contained in $T_{i, 1}$. 
    \newline
    \indent
    \textbf{Inductive Step.} It will be shown that there exists a splitting bigon $B_{i + 1}$ such that $T_{i + 1, 1}$ contains the associated loop of exactly $i$ splitting bigons. Furthermore, $i$ of these bigons can be identified through the inductive hypothesis, namely $B_1, B_2, \ldots, B_{i}$. Thus, it remains to show that the associated loop of no other splitting bigon is contained in $T_{i + 1, 1}$.
    \newline
    \indent
    Suppose, for the sake of contradiction, that no such $B_k$ exists. Then, for all splitting bigons $X$ that split the doubled slit torus into $T_1$ and $T_2$ with $T_1$ containing the red hole, there exists a splitting bigon $X_1 \neq B_j$ for any $1 \leq j \leq i$ whose associated loop lies completely in $T_1$. Similarly, applying this idea for $X_1$, it is deduced that there must exist an $X_2 \neq B_j$ for any $1 \leq j \leq i$ whose associated loop is even closer to the red hole. Continuing this ad infinitum, there must exist an infinite sequence $X_1, X_2, \ldots$ of distinct splitting bigons, which contradicts the finiteness of the total number of splitting bigons. This completes the inductive step and thus the proof of the lemma.
\end{proof}
This idea of ordering the bigons can be used as follows.
\begin{prop}\label{split}
    The splitting bigons $B_1, B_2, \ldots, B_k$, as per Lemma \ref{order}, divide the doubled slit torus into $k + 1$ surfaces with disjoint interiors, namely two slit tori $\Gamma_1'$, $\Gamma_2'$, and genus zero surfaces $S_i$ for $1 \leq i < k$ and $S_i$ such that $S_i$ is bounded by the bigons $B_i$ and $B_{i + 1}$ for all $i$.
\end{prop}
\begin{proof}
    Suppose that the doubled slit torus is cut along $B_i$ and $B_{i + 1}$ for some $1 \leq i < k$. Then three surfaces $\mathcal{S}_1, \mathcal{S}_2, \mathcal{S}_3$ are formed such that $\mathcal{S}_1$ contains the red hole, $\mathcal{S}_2$ has genus zero, and $\mathcal{S}_3$ contains the blue hole, with $B_i$ forming the boundary between $\mathcal{S}_1$ and $\mathcal{S}_2$ and $B_{i + 1}$ forming the boundary between $\mathcal{S}_2$ and $\mathcal{S}_3$. Let $S_i = \mathcal{S}_2$.
    \newline
    \indent
    By Lemma \ref{order}, all the surfaces of the form $S_i$ must have disjoint interiors as there exists no splitting bigon whose associated loop is contained in $S_i$. 
    \newline
    \indent
    Now, suppose that $B_1$ divides the doubled slit torus into $T_{1, 1}$ and $T_{1, 2}$, where $T_{1, 1}$ contains the red hole; similarly, $B_k$ divides the doubled slit torus into $T_{k, 1}$ and $T_{k, 2}$, where $T_{k, 1}$ contains the red hole. Then, after cutting along all the splitting bigons $B_1, B_2, \ldots, B_k$, the surfaces $S_i$ are formed as indicated previously, as well as $\Gamma_1' = T_{1, 1}$ (which contains the red hole) and $\Gamma_2' = T_{k, 2}$ (which contains the blue hole). This completes the proof of the proposition.
\end{proof}
\begin{remark}
    The preceding results characterize the structure of all the splitting bigons in a doubled slit torus. By Lemma \ref{order}, it can be seen that they can be ordered in a ``line" from one hole to the other and by Lemma \ref{split}, it can be seen that adjacent splitting bigons bound a surface with genus zero.
\end{remark}
The following lemma can also be shown, which will help to establish some of our uniqueness results.
\begin{lem}\label{trianglesplit}
Let $T$ be a triangulation on a doubled slit torus, and let a splitting bigon consisting of vertices $\mathcal{V}$ and $\mathcal{V'}$ divide the surface into two slit tori $\Gamma_1$ and $\Gamma_2$. Then there exists a unique way to determine an unordered pair of subgraphs $\{\mathcal{T}, \mathcal{T'}\}$ such that $\mathcal{T}$ triangulates slit torus $\Gamma_1$, $\mathcal{T'}$ triangulates slit torus $\Gamma_2$, and $\mathcal{T} \cup \mathcal{T'} = T$.
\end{lem}
\begin{proof}
Both $\mathcal{V}$ and $\mathcal{V'}$ must belong to both $\mathcal{T}, \mathcal{T'}$. The edges, if any, connecting $\mathcal{V}$ to $\mathcal{V'}$, not contained in the bigon, as well as any loops containing just $\mathcal{V}$ or just $\mathcal{V'}$, are each contained in exactly one of $\Gamma_1$ or $\Gamma_2$. Otherwise, they must pass through $\mathcal{V}$ or $\mathcal{V'}$ or an additional third vertex on the bigon, neither of which are possible by assumption.
\newline
\indent
It needs to be decided which of $\mathcal{T}$ and $\mathcal{T'}$ is to triangulate which of $\Gamma_1, \Gamma_2$. Without loss of generality, assume that $\mathcal{T}$ is to triangulate $\Gamma_1$ and $\mathcal{T'}$ is to triangulate $\Gamma_2$. Since it may be the other way around, the pair is described as being unordered.
\newline
\indent
If such an edge or loop is contained in $\Gamma_1$, then it must be placed in $\mathcal{T}$ and not in $\mathcal{T'}$. If such an edge or loop is contained in $T_2$, then it must be placed in $\mathcal{T'}$ and not in $\mathcal{T}$.
Lastly, all other vertices, edges, and loops are contained in exactly one of the two connected components, and so must belong to $\mathcal{T}$ if they are part of $\Gamma_1$ and $\mathcal{T'}$ otherwise. Therefore, there is only one possible way to divvy up $T$ into $\mathcal{T}$ and $\mathcal{T'}$. This completes the proof of the lemma.
\end{proof}
These results can now be generalized to surfaces of higher genus.
\begin{definition}
    A translation surface with genus $g \geq 2$ is called \textit{slitted} if it belongs in the stratum $\mathcal{H}(1, 1, \ldots, 1, 1)$ and is composed of $g$ different tori (with up to two slits) connected to one another sequentially.
\end{definition}
The definition of a splitting bigon can be generalized to an arbitrary slitted translation surface. 
\begin{definition}
    A \textit{splitting bigon} $B$ in a slitted translation surface $S$ is a bigon such that if one cuts along the associated loop of $B$, the surfaces $S_1 \cup S_2 = S$ that are formed each have genus at least one.
\end{definition}
\begin{remark}
    A doubled slit torus is an example of a slitted translation surface, and the definition of a splitting bigon over the doubled slit torus matches the generalized definition.
\end{remark}
\begin{lem}
    Let $S$ be an arbitary slitted translation surface, and suppose that $B_1$ and $B_2$ are both splitting bigons of $S$. Then if $B_1$ divides $S$ into surfaces $S_1$ and $S_2$, every point on the associated loop of $B_2$ lies on $S_i$ for some choice $i \in \{1, 2\}$.
\end{lem}
\begin{remark}
    The proof of this result is almost the same as that of Lemma \ref{lemma2}.
\end{remark}
To characterize these splitting bigons, consider the following two lemmas, which generalize Lemma \ref{order} and Lemma \ref{split}.
\begin{lem}\label{genorder}
    For some sequence of positive integers $k_1, k_2, \ldots, k_{g-1}$, the splitting bigons of a slitted translation surface $S$ can be ordered as $B_{1, 1}, B_{2, 1}, \ldots, B_{k_1, 1}, B_{1, 2}, B_{2, 2}, \ldots, B_{k_2, 2}, B_{1, 3}, \allowbreak \ldots, B_{k_{g-1}, g - 1}$ such that $B_{p, i}$ represents a bigon whose associated loop splits $S$ into a genus $i$ surface $S_1$ and a genus $g - i$ surface $S_2$ such that all $B_{x, y}$ where either $y < i$ or both $x < p$ and $y = i$ lie in $S_1$ and all remaining other bigons lie in $S_2$.
\end{lem}
\begin{proof}
    The result is shown by induction on $g \geq 2$.
    \newline
    \indent
    \textbf{Base Case.} The case of $g = 2$ yields $g - 1 = 1$, and so there is a single value $k_1$ for which the bigons are ordered $B_{1, 1}, B_{2, 1}, \ldots, B_{k_1, 1}$ such that when $S$ is divided along $B_{i, 1}$, bigons of the form $B_{j, 1}$ for $j < i$ lie on one slit torus, and bigons of the form $B_{j, 1}$ for $j > i$ lie on the other slit torus. Therefore, the result follows from Lemma \ref{order}.
    \newline
    \indent
    \textbf{Inductive Hypothesis.} Suppose that for some $g - 1$, the splitting bigons in all slitted translation surfaces $S'$ of genus $g - 1$ can be numbered as $B_{1, 1}, B_{2, 1}, \ldots, B_{k_1, 1}, B_{1, 2}, B_{2, 2}, \ldots, \allowbreak B_{k_2, 2}, B_{1, 3}, \ldots, B_{k_{g-2}, g - 2}$ such that $B_{p, i}$ represents a bigon whose associated loop splits $S'$ into a genus $i$ surface $S_1'$ and a genus $g - 1 - i$ surface $S_2'$ such that all $B_{x, y}$ where either $y < i$ or both $x < p$ and $y = i$ lie in $S_1'$ and all remaining other bigons lie in $S_2'$.
    \newline
    \indent
    \textbf{Inductive Step.} The result will be proven for a slitted translation surface $S$ of genus $g$. The idea is to consider a subsurface $S'$ of $S$ consisting of the first $g - 1$ slit tori connected together. Then $S'$ is a slitted translation surface of genus $g - 1$ such that the splitting bigons of $S'$ are all splitting bigons of $S$. By the inductive hypothesis, these can be numbered as $B_{1, 1}, B_{2, 1}, \ldots, B_{k_1, 1}, B_{1, 2}, B_{2, 2}, \ldots, B_{k_2, 2}, B_{1, 3}, \ldots, B_{k_{g-2}, g - 2}$ subject to the given condition. 
    \newline
    \indent
    Now, consider the remaining $k_{g - 1}$ splitting bigons of $S$ that are not part of $S'$. Cutting along any of these bigons results in an $S_2$ surface of genus $1$. The hole in the slit torus $S_2$ can be colored red and the bigons $B_{1, g-1}, B_{2, g- 1}, \ldots, B_{k_{g-1}, g-1}$ can be ordered based upon their proximity to the red hole, where $B_{1, g- 1}$ is the farthest and $B_{k_{g - 1}, g- 1}$ is the closest (this is possible by using the same method as in the proof of Lemma \ref{order}). 
    \newline
    \indent
    Then, consider an arbitrary splitting bigon of $S$ denoted by $B_{p, i}$ that splits $S$ into $S_1$ and $S_2$, where $S_1$ has genus $i$ and $S_2$ has genus $g - i$. Suppose first that $i = g - 1$. Then all bigons of the form $B_{x, y}$ with $y < g - 1$ will clearly lie in $S_1$, and by the ordering of bigons of the form $B_{x, g - 1}$ stated above, all $x < p$ and $y = g - 1$ will also lie in $S_1$. However, all bigons $B_{x, y}$ with $x > p$ and $y = g - 1$ will lie in $S_2$. Now, suppose that $i < g - 1$. The ordering then works by the inductive hypothesis for all bigons of the form $B_{x, y}$ with $y \neq g - 1$. But the bigons with $y = g - 1$ also must all lie in $S_2$, since they are not splitting bigons of $S'$. Therefore, the ordering satisfies the given condition, completing the inductive step.
\end{proof}
\begin{lem}\label{gensplit}
    The splitting bigons $\{B_{x, y}, x \leq k_y\}$, as per Lemma \ref{genorder}, divide the slitted translation surface into $1 + \sum \limits_{i = 1}^{g - 1}k_i$ surfaces with disjoint interiors, each of which is bounded by at most two splitting bigons and has genus either zero or one.
\end{lem}
\begin{proof}
    Suppose that the slitted translation surface is cut along two consecutive bigons $B_{p, q}$ and $B_{r, s}$ in the ordering given by Lemma \ref{genorder}, such that both $s = q$ and $r = p + 1$ or $s = q + 1$, $p = k_q$, $r = 1$. Then three surfaces $\mathcal{S}_1, \mathcal{S}_2, \mathcal{S}_3$ are formed such that $\mathcal{S}_1$ has genus $q$, $\mathcal{S}_3$ has genus $g - s$, $\mathcal{S}_2$ has genus $s - q$, and $\mathcal{S}_1 \# \mathcal{S}_2 \# \mathcal{S}_3 = S$.
    \newline
    \indent
    But $s = q$ or $s = q + 1$, and so $s - q \in \{0, 1\}$. Therefore, surfaces of the form $\mathcal{S}_2$ above necessarily have genus either zero or one. Consider all of the surfaces of the form $\mathcal{S}_2$ and number them in order $S_1, S_2, \ldots, S_{k - 1}$ where $k = \sum \limits_{i = 1}^{g - 1} k_i$ is the total number of splitting bigons. By Lemma \ref{order}, all the surfaces of the form $S_i$ must have disjoint interiors as there exists no splitting bigon whose associated loop is contained in $S_i$. 
    \newline
    \indent
    Now, suppose that $B_{1, 1}$ divides the doubled slit torus into a surface $T$ of genus one and another surface of $g - 1$; similarly, $B_{k_{g - 1}, g - 1}$ divides the doubled slit torus into a surface of $g - 1$ and another surface $T'$ of genus one. Then, after cutting along all the splitting bigons, the surfaces $S_i$ are formed as indicated previously, as well as $T$ and $T'$. Therefore, there are $k - 1 + 2 = k + 1$ total surfaces with disjoint surfaces, each bounded by at most two bigons, whose connected sum is $S$, and each of which has genus either zero or one. This completes the proof of the proposition.
\end{proof}

\section{Uniqueness Results}\label{main}
Suppose that there exists a circle packing $P$ on the doubled slit torus with an associated triangulation $T$. Furthermore, suppose that the packing contains two double circles $C_1$ and $C_2$ such that the slit connects the centers of the two circles. 
\begin{figure}[H]

    \centering
    \includegraphics[width=10cm]{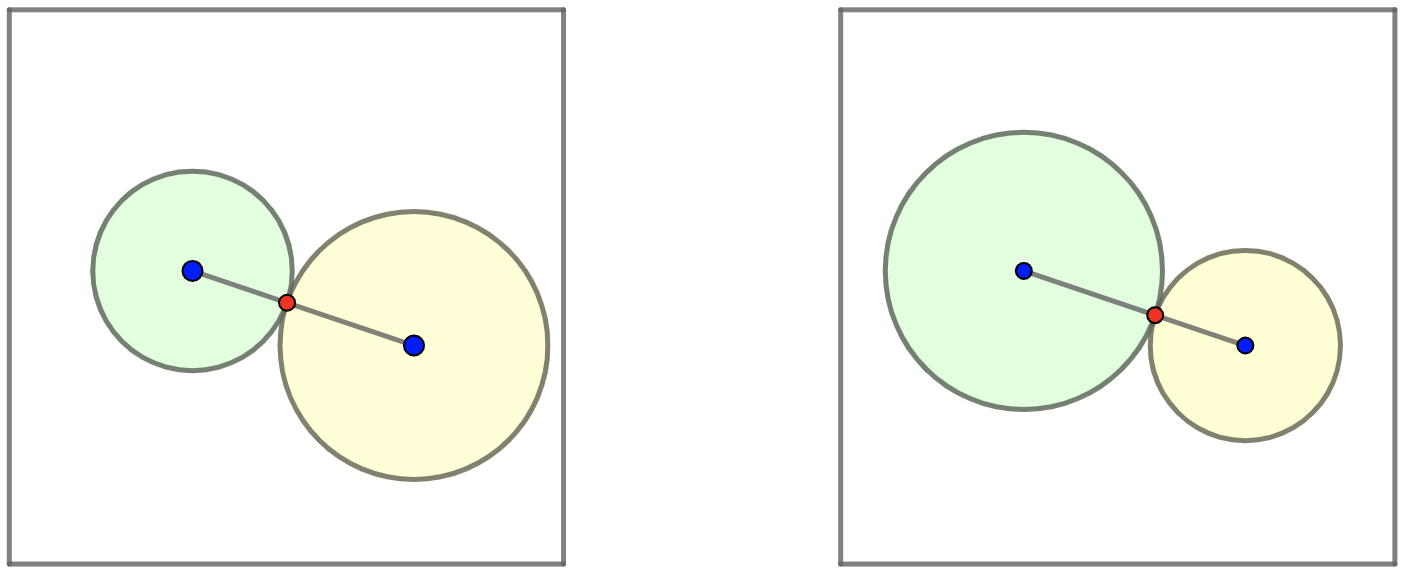}
    \caption{An illustration of the circle packing $P$. Here, $C_1$ is the green double circle and $C_2$ is the yellow double circle. The two copies of each double circle on the different slit tori have differing radii. The circles $C_1$ and $C_2$ are tangent at two distinct points, which are colored red.}
    \label{fig:figure12}
\end{figure}
\begin{remark}
Note that the two copies of $C_1$ on the different slit tori need not have the same radius, as illustrated in Figure \ref{fig:figure12}. Similarly, the two copies of $C_2$ need not have the same radius, only the same center.
\end{remark}
Note that the bigon connecting the vertices corresponding to $C_1$ and $C_2$ must be splitting since the two double circles topologically divide the doubled slit torus into two slit tori. Number the $k$ splitting bigons as $B_1, B_2, \ldots, B_k$ as per Lemma \ref{order}.
\begin{definition}
    For a circle packing $Q$ containing $C_1$ and $C_2$ (which are fixed in place) of the same doubled slit torus with the same associated topological triangulation $T$, the \textit{order} $i$ of $Q$ is defined to be the unique $1 \leq i \leq k$ such that the bigon $B_i$ consists of the vertices $v_1$ and $v_2$, corresponding to $C_1$ and $C_2$, respectively, in $Q$.
\end{definition}
Consider the following lemma.
\begin{lem}\label{samesplit}
    Let $P'$ be a circle packing containing $C_1$ and $C_2$ of the doubled slit torus on the same triangulation $T$. If the order of $P$ is equal to the order of $P'$, then $P'$ is equal to $P$ or is equal to a hyperelliptic involution of $P$.
\end{lem}
\begin{proof}
    Denote the shared order of both packings as $i$ and let $V_1, V_2 \in T$ be the vertices corresponding to circles $C_1$ and $C_2$ respectively in both circle packings. Then $V_1$ and $V_2$ make up the bigon $B_i$ in $T$. 
    \newline
    \indent
    After removing $V_1$ and $V_2$, the edges between them, and the loops containing just $V_1$ or just $V_2$, exactly two connected components are left by Lemma \ref{lemma1}. As before, an equivalence relation $\sim_B$ is set up on all the vertices of $T$ except for $V_1$ and $V_2$ with two equivalence classes such that $X \sim_i Y$ holds if and only if there exists a path from $X$ to $Y$ along $T$ passing through neither $V_1$ nor $V_2$.
    \newline
    \indent
    Now, suppose that $B_i$ divides the doubled slit torus into two tori $T_{i, 1}$ and $T_{i, 2}$. By Lemma \ref{trianglesplit}, there exists a unique unordered pair of subgraphs $\{\mathcal{T}, \mathcal{T'}\}$ such that $\mathcal{T}$ triangulates $T_{i, 1}$ in circle packing $P$ and $\mathcal{T'}$ triangulates $T_{i, 2}$ in circle packing $P$, and $\mathcal{T} \cup \mathcal{T'} = T$. 
    \newline
    \indent
    In circle packing $P'$ then, either $\mathcal{T}$ triangulates $T_{i, 1}$ and $\mathcal{T'}$ triangulates $T_{i, 2}$, or $\mathcal{T'}$ triangulates $T_{i, 1}$ and $\mathcal{T}$ triangulates $T_{i, 2}$. Let $T_1$ be the element of $\{\mathcal{T}, \mathcal{T'}\}$ that actually triangulates $T_{i, 1}$ in $P'$ and let $T_2$ be the other element.
    \newline
    \indent
    Extend the surfaces $T_{i, 1}$ and $T_{i, 2}$ (which are both slit tori) to tori $\overline{T_{i, 1}}$ and $\overline{T_{i, 2}}$ by adding in a surface of genus zero. Then $T_1$ triangulates $\overline{T_{i, 1}}$ and $T_2$ triangulates $\overline{T_{i, 2}}$, as there is simply an extra bigon face being added and none of the other conditions for triangulation are broken.
    \newline
    \indent
    Now, separate circle packings $P'_1$ and $P'_2$ can be considered on $\overline{T_{i, 1}}$ and $\overline{T_{i, 2}}$ with associated triangulations $T_1$ and $T_2$ respectively. Note that both $P_1'$ and $P_2'$ contain the two circles $C_1$ and $C_2$, which are fixed in place on the doubled slit torus. By Corollary \ref{dut},  it is known that the remaining circles on each of the extended tori (and thus the slit tori as well) must be fixed in place. This implies that the entire circle packing is fixed. Therefore, the only way to vary the entire circle packing $P'$ is to swap $T_1$ as being $\mathcal{T}$ or $\mathcal{T'}$, which is exactly the hyperelliptic involution, as desired. This completes the proof of the proposition.
\end{proof}
This lemma can be used to prove the following theorem, our first main uniqueness result.
\begin{thm}\label{unique}
    Suppose that a circle packing $P$ with an associated triangulation $T$ is fixed on an $\mathcal{H}(1, 1)$ doubled slit torus that contains two externally tangent double circles $C_1$ and $C_2$, such that the slit connects the centers of these two circles. Then, there are only a finite number of possibilities for a second, distinct circle packing $P'$ with the triangulation $T$ containing $C_1$ and $C_2$ (which are fixed in place).    
\end{thm}
\begin{proof}
    Consider the vertices $v_1$ and $v_2$ in $T$ corresponding to $C_1$ and $C_2$, respectively, as per the circle packing $P'$. Then, there is a splitting bigon determined by $v_1$ and $v_2$. Suppose that this splitting bigon is of the form $B_i$, where $1 \leq i \leq k$. Then there are $k$ possibilities for $i$. 
    \newline
    \indent
    Suppose the value of $i$ is fixed. By Proposition \ref{samesplit}, there are at most two possibilities for $P'$. Therefore, there are at most $2k$ possibilities for the circle packing $P'$, of which $2k - 1$ are distinct from the original packing $P$. This is a finite number, as desired.
\end{proof}
\begin{definition}
    Define the number of splitting bigons in a given triangulation $T$ to be $d(T)$. 
\end{definition}
\begin{remark}\label{count}
    The proof of the previous theorem establishes an upper bound of $2 \cdot d(T) - 1$ possibilities for $P' \neq P$ also containing $C_1$ and $C_2$. A lower bound for the number of possibilities for $P' \neq P$ is zero, achieved when $d(T) = 1$ and the subgraphs $\mathcal{T}$ and $\mathcal{T'}$ are equal.
\end{remark}
This uniqueness result can be generalized to surfaces with genus $g > 2$, using Lemma \ref{genorder} and Lemma \ref{gensplit}. This requires the following definition.
\begin{definition}
    For a circle packing $Q$ of the same slitted translation surface with the same triangulation $T$ containing circles of the form $C_{j, 1}$ and $C_{j, 2}$ (which are fixed in place along the $j$th slit in sequential order, which connects the centers of the two circles), the \textit{order} $i$ of $Q$ with respect to the $j$th slit is defined to be the unique $1 \leq i \leq k_j$ such that the bigon $B_{i, j}$ consists of the vertices $v_1$ and $v_2$, corresponding to $C_{j, 1}$ and $C_{j, 2}$, respectively, in $Q$.
\end{definition}
Suppose that there exist two circle packings $P$ and $P'$ on the same slitted translation surface each containing externally tangent double circles $C_{j, 1}$ and $C_{j, 2}$, which are fixed in place on the $j$th slit (that connects the centers of the two circles). Suppose further that $P$ and $P'$ have the same associated topological triangulation $T$. Consider the following lemma.
\begin{lem}\label{euler}
    Let $B$ and $B'$ be two arbitrary splitting bigons in the shared triangulation $T$, whose associated loops bound surfaces $S_P$ and $S_P'$ with respect to the packings $P$ and $P'$, respectively. Then, the genus of $S_P$ is equal to the genus of $S_P'$.
\end{lem}
\begin{proof}
    Consider any circle packing $Q$ on $T$ and the surface $S_Q$ bounded by the associated loops of bigons $B$ and $B'$. It suffices to show that the genus of $S_Q$ is invariant. Let $v_1$ and $v_2$ correspond to the vertices of $B$ and $v_3$ and $v_4$ correspond to the vertices of $B'$ (note that $v_1, v_2, v_3, v_4$ need not all be distinct). Then consider all vertices and edges on a path in $T$ from a vertex $v_i$ to $v_j$, where $i \in \{1, 2\}$ and $j \in \{3, 4\}$. Along with $B$ and $B'$, they form a subgraph $T'$ in $T$. This subgraph $T'$ forms a triangulation of $S_Q$ is invariant of the actual packing $Q$. Therefore, the Euler characteristic of $T'$ is invariant of $Q$. But since $T'$ triangulates $S_Q$, the Euler characteristic of $S_Q$ is invariant. It follows that the genus of $S_Q$ is invariant, since $\chi(S_Q) = 2 - 2g(S_Q)$. This completes the proof of the lemma.
\end{proof}
This idea can be used to show the following theorem.
\begin{thm}\label{genunique}
    Suppose that $S$ is a genus $g$ slitted translation surface with $g > 1$ and that $P$ is a circle packing containing $C_{j, 1}$ and $C_{j, 2}$ (for all $1 \leq j \leq g - 1$), such that the double circles are fixed in place along the $j$th slit sequentially, which connects their centers. Then, there are only a finite number of circle packings $P' \neq P$ on $S$ with the same associated triangulation $T$ that also contain all the double circles of the form $C_{j, 1}$ and $C_{j, 2}$ fixed in place along each slit.
\end{thm}
\begin{proof}
    Number the splitting bigons of $P$ in order as $\{B_{x, y}, y \leq k_x\}$ for $k_1, k_2, \ldots, k_{g - 1}$. Consider the splitting bigon corresponding to the first slit in $P'$. Since this bigon divides $S$ into a surface of genus one and a surface of genus $g - 1$, there are two cases.
    \newline
    \indent
    \textbf{Case 1.} The splitting bigon corresponding to the first slit in $P'$ is of the form $B_{x, 1}$ for $x \leq k_1$. It will be shown that the bigon corresponding to the $j$th slit in $P'$ must be of the form $B_{x, j}$ for $x \leq k_j$ via induction on $j$.
    \newline
    \indent
    \textbf{Base Case.} If $j = 1$, the result is clear by the case assumption.
    \newline
    \indent
    \textbf{Inductive Hypothesis.} Suppose that the splitting bigon corresponding to the $j$th slit in $P'$ is of the form $B_{x, j}$ for some $j < g - 1$. 
    \newline
    \indent
    \textbf{Inductive Step.} Let $B$ be the splitting bigon corresponding to the $j + 1$th slit in $P'$. By Lemma \ref{euler}, $B$ needs to correspond to a bigon in $P$ that bounds a surface of genus one with the splitting bigon $B_{x, j}$ corresponding to the $j$th slit. Therefore, $B = B_{y, j - 1}$ for some $y$ or $B = B_{y, j + 1}$ for some $y$. However, splitting $S$ along the associated loop of $B$ must result in a genus $j + 1$ surface and a genus $g - 1 - j$ surface as $B$ corresponds to the $j + 1$th slit. Therefore, $B = B_{y, j - 1}$ is impossible, forcing $B = B_{y, j + 1}$, which completes the inductive step. 
    \newline
    \indent
    Therefore, there are $\prod \limits_{i = 1}^{g - 1} k_i$ ways to choose which splitting bigons correspond to which slit in this case.
    \newline
    \indent
    \textbf{Case 2.} The splitting bigon corresponding to the first slit in $P'$ is of the form $B_{x, g - 1}$ for $x \leq k_{g - 1}$. It will be shown that the bigon corresponding to the $j$th slit in $P'$ must be of the form $B_{x, g - j}$ for $x \leq k_{g - j}$ via induction on $j$.
    \newline
    \indent
    \textbf{Base Case.} If $j = 1$, the result is clear by the case assumption.
    \newline
    \indent
    \textbf{Inductive Hypothesis.} Suppose that the splitting bigon corresponding to the $j$th slit in $P'$ is of the form $B_{x, g - j}$ for some $j < g - 1$. 
    \newline
    \indent
    \textbf{Inductive Step.} Let $B$ be the splitting bigon corresponding to the $j + 1$th slit in $P'$. By Lemma \ref{euler}, $B$ needs to correspond to a bigon in $P$ that bounds a surface of genus one with the splitting bigon $B_{x, g - j}$ corresponding to the $j$th slit. Therefore, $B = B_{y, g + 1 - j}$ for some $y$ or $B = B_{y, g - 1 - j}$ for some $y$. However, splitting $S$ along the associated loop of $B$ must result in a genus $j + 1$ surface and a genus $g - 1 - j$ surface as $B$ corresponds to the $j + 1$th slit. Therefore, $B = B_{y, g + 1 - j}$ is impossible, forcing $B = B_{y, g - 1 - j}$, which completes the inductive step. 
    \newline
    \indent
    Note that every assignment of a splitting bigon to a slit in this case constitutes a topological reversal of the resulting triangulation as per the assignment in the previous case, and vice versa. Therefore, there are \textit{at most} $\prod \limits_{i = 1}^{g - 1} k_i$ ways to choose which splitting bigons correspond to which slit in this case (with there being zero new ways if the triangulation is topologically invariant under reversal).
    \newline
    \indent
    Once the splitting bigons corresponding to each slit in $P'$ are determined explicitly, subgraphs $T_1, T_2, \ldots, T_g$ can be formed such that $T_i$ triangulates the $i$th slit torus sequentially in order and $\bigcup \limits_{i = 1}^{g} T_i = T$.  
    \newline
    \indent
    Note that $T_2, T_3, \ldots, T_{g - 1}$ are determined uniquely. For given splitting bigons $B$ and $B'$ that correspond to the consecutive $i$th slit and $i + 1$th slit, consider all the vertices and edges of paths from a vertex in $B$ to a vertex in $B'$, in addition to the vertices and edges of the bigons themselves, as in the proof of Lemma \ref{euler}. This creates a unique triangulation $T_{i + 1}$. This process works for all $1 \leq i < g - 1$.
    \newline
    \indent
    Suppose all the splitting bigons and elements of $T_2, \ldots, T_{g - 1}$ from $T$ are eliminated. Two unordered connected components, denoted $T_1$ and $T_g$, are left, both of which triangulate a slit torus. Note that $T_1$ and $T_g$ can both be swapped if they are different. Suppose that they are and there exist two distinct triangulations $\mathcal{T}$ and $\mathcal{T'}$ of a genus one slit torus such that $T_1 = \mathcal{T}, T_g = \mathcal{T'}$ or $T_1 = \mathcal{T'}, T_g = \mathcal{T}$. Since:
    \[\mathcal{T} \cup T_2 \cup T_3 \cup \cdots \cup T_{g - 1} \cup \mathcal{T'} = T = \mathcal{T'} \cup T_2 \cup T_3 \cup \cdots \cup T_{g - 1} \cup \mathcal{T}\]
    and $\mathcal{T}$ is distinct from $\mathcal{T'}$, we have that $T_2 \cup T_3 \cup \ldots \cup T_{g - 1}$, a triangulation of the connected sum of the $g - 2$ tori in the middle, must be topologically invariant under reversal. Therefore, the two embeddings of $T$ in $S$ formed by swapping $\mathcal{T}$ and $\mathcal{T'}$ must be reversals of each other. But the duplicity of reversals has already been taken care of in Case $2$ above, so these do not create more possibilities for $P'$ that have not already been counted while determining the assignment of splitting bigons to slits.
    \newline
    \indent
    Consider the slit tori $\Gamma_1, \Gamma_2, \ldots, \Gamma_g$ such that $\Gamma_1 \# \Gamma_2 \# \ldots \# \Gamma_g = S$. Suppose now that there exists a tuple $(T_1, T_2, \ldots, T_g)$ such that $T_i$ triangulates $\Gamma_i$. For each $i$, the slit torus $\Gamma_i$ can be extended to the torus $\overline{\Gamma_i}$ by adding in either one or two surfaces of genus zero. Then $T_i$ triangulates $\Gamma_i$ for all $1 \leq i \leq g$, as there are simply either one or two extra bigon faces being added and none of the other conditions for triangulation are broken.
    \newline
    \indent
    Now, consider separate circle packings $P'_1, P'_2, \ldots, P'_g$ aon $\Gamma_1, \Gamma_2, \ldots, \Gamma_g$, respectively, with associated triangulations $T_1, T_2, T_3, \ldots, T_g$ respectively. Note that each packing $P'_i$ contains either two double circles or four double circles that are fixed in place. By Corollary \ref{dut}, it is known that the remaining circles on each of the extended tori (and thus the slit tori as well) must be fixed in place. This implies that the entire circle packing is fixed.
    \newline
    \indent
    Since there are at most $2 \cdot \prod \limits_{i = 1}^{g - 1} k_i$ ways to choose which splitting bigons correspond to which slit as well as an orientation with respect to reversal (as per the casework), there are at most $2 \cdot \prod \limits_{i = 1}^{g - 1} k_i$ possibilities for the circle packing $P'$ since there is exactly one tuple $(T_1, T_2, T_3, \ldots, T_g)$ up to orientation with respect to reversal. One of these possibilities is $P' = P$, so there are at most $2 \cdot \prod \limits_{i = 1}^{g - 1} k_i - 1$ other possibilities for $P' \neq P$. Since this bound is finite, the proof of the theorem is complete. 
\end{proof}
\begin{remark}\label{gencount}
    If $g = 2$ is plugged into the upper bound of $2 \cdot \prod \limits_{i = 1}^{g - 1} k_i - 1$ derived in the preceding proof, Theorem \ref{unique} is recovered (in particular, the bound in Remark \ref{count}).
\end{remark}
\section{Research Directions}\label{futdir}
The goals of this research project can be put into two categories: \textit{existence} and \textit{uniqueness}. Existence refers to showing that a circle packing exists for an arbitrary triangulation, under certain conditions. Uniqueness refers to showing that no two circle packings can have the same contacts graph and be fundamentally distinct with respect to the surface. 
\newline
\indent
In this paper, we have made some progress towards showing the uniqueness of certain types of circle packings on translation surfaces. Below are some ideas regarding next steps for continuing to investigate this.
\begin{itemize}
    \item Investigating uniqueness for circle packings on doubled slit tori with $k$ double circles $C_1, C_2, \ldots, C_k$, as in the proof of Proposition \ref{triangprop}, for which each adjacent pair of which is externally tangent, fixed in place.
    \item Investigating uniqueness for circle packings in the $\mathcal{H}(2)$ stratum, where the slit has exactly one endpoint. More generally, investigating uniqueness for circle packings on a surface with a cone point of order greater than $1$.
    \item Investigating uniqueness for circle packings of arbitrary square-tiled surfaces as in \ref{squaredef}. This has the added benefit that many non-squared-tiled translation surface can be thought of as deformations of square-tiled ones as in Remark \ref{square}.
\end{itemize}
One of the entry points to existence is to answer the following question, extending the Koebe-Andreev-Thurston theorem to genus $2$ translation surfaces.
\begin{question}
Given an arbitrary triangulation $T$ of a genus $2$ translation surface $M$, can one always find a circle packing of some $M'$ with contacts graph $T$ such that $M$ and $M'$ lie in the same stratum?
\end{question}
A similar question can be posed regarding square tilings on such surfaces, which are in many respects quite similar to circle packings.
\begin{question}
Given an arbitrary translation surface $M$, can one always find a square tiling of some other translation surface formed by applying an affine transformation to $M$?
\end{question}
This was inspired by a result by Cong (see \cite{cong}), showing that while the translation surface formed by a regular octagon can never be square-tiled, there exists an affine transformation such that the resulting surface can be square-tiled. 
\newline
\indent
Another question regarding the uniqueness of circle packings is whether or not the upper bounds described in \S\ref{main} can always be achieved.
\begin{question}
Does there exist a distinct circle packing of every order $1 \leq i \leq k$, as in the proof of Theorem \ref{unique}? 
\end{question}
This question can also be generalized to genus $g > 2$, as follows.
\begin{question}
Does there exist a distinct circle packing of every tuple of orders $(x_1, x_2, \ldots, x_{g - 1})$ with respect to each of the $g - 1$ slits sequentially where $1 \leq x_i \leq k_i$ for all $1 \leq i \leq g - 1$, as in the proof of Theorem \ref{genunique}? 
\end{question}
\section*{Acknowledgements}
I would like to sincerely thank my MIT PRIMES-USA mentor, Professor Sergiy Merenkov, for his immeasurable guidance and assistance throughout the reading and research periods. I am grateful to Professor W. Patrick Hooper for providing a compilation of useful resources on translation surfaces. I would also like to thank Dr. Tanya Khovanova for helpful feedback on preliminary reports. Additionally, I would like to thank the organizers of the MIT PRIMES-USA program, Dr. Slava Gerovitch and Professor Pavel Etingof, for creating this amazing opportunity to conduct mathematical research. 
\bibliographystyle{plain}
\bibliography{refs}
\nocite{*}
\end{document}